\documentclass[11pt]{article}
\usepackage[margin=2.4cm,centering]{geometry}
\usepackage{amsmath, amsthm, amssymb}
\usepackage{graphicx}
\usepackage[small,it]{caption}
\usepackage{tikz}
\usetikzlibrary{calc}
\usepackage[T1]{fontenc}
\usepackage[sc]{mathpazo}
\usepackage{parskip}
\definecolor{darkblue}{rgb}{0, 0, .4}
\usepackage[bookmarks]{hyperref}
\hypersetup{
    colorlinks=true,
    linkcolor=darkblue,
    anchorcolor=darkblue,
    citecolor=darkblue,
    urlcolor=darkblue,
    pdfpagemode=UseThumbs,
    pdftitle={Well-quasi-ordering and finite distinguishing number},
    pdfsubject={Combinatorics},
    pdfauthor={A. Atminas and R. Brignall},
    pdfkeywords={well-quasi-ordering,hereditary property,graph}
}

\title{Well-quasi-ordering and finite distinguishing number}
\author{Aistis Atminas\\[10pt]
\small Department of Mathematics\\
\small London School of Economics\\
\small London, WC2A 2AE\\
\small United Kingdom
\and
Robert Brignall\\[10pt]
\small Department of Mathematics and Statistics\\
\small The Open University\\
\small Milton Keynes, MK7 6AA\\
\small United Kingdom}
\date{May 3, 2019}

\newcommand{\A}{\mathcal{A}}
\newcommand{\B}{\mathcal{B}}
\newcommand{\C}{\mathcal{C}}
\newcommand{\X}{\mathcal{X}}
\newcommand{\F}{\mathcal{F}}
\newcommand{\FF}{\mathfrak{F}}
\renewcommand{\P}{\mathcal{P}}
\DeclareMathOperator{\free}{Free}

\begingroup
    \makeatletter
    \@for\theoremstyle:=definition,remark,plain\do{%
        \expandafter\g@addto@macro\csname th@\theoremstyle\endcsname{%
            \addtolength\thm@preskip\parskip
            }%
        }
\endgroup

\theoremstyle{plain}
\newtheorem{theorem}{Theorem}[section]
\newtheorem{question}[theorem]{Question}
\newtheorem{lemma}[theorem]{Lemma}
\newtheorem{definition}[theorem]{Definition}

\newtheorem{conjecture}[theorem]{Conjecture}

\newtheorem{corollary}[theorem]{Corollary}
\newtheorem{problem}[theorem]{Problem}
\newtheorem{algo}[theorem]{Algorithm}
\newtheorem{proposition}[theorem]{Proposition}

\usepackage{paralist}

\long\def\probl#1#2{~\par
\begin{compactitem}
\item[\textsc{Input:}]
#1
\item[\textsc{Output:}]
#2
\end{compactitem}}

\tikzstyle{every node}=[circle, draw, fill=black,
                        inner sep=0pt, minimum width=4pt]

\begin{document}
\maketitle
\begin{abstract}
Balogh, Bollob\'as and Weinreich showed that a parameter that has since been termed the \emph{distinguishing number} can be used to identify a jump in the possible speeds of hereditary classes of graphs at the sequence of Bell numbers.  

We prove that every hereditary class that lies above the Bell numbers and has finite distinguishing number contains a boundary class for well-quasi-ordering. This means that any such hereditary class which in addition is defined by finitely many minimal forbidden induced subgraphs must contain an infinite antichain. As all hereditary classes below the Bell numbers are well-quasi-ordered, our results complete the answer to the question of well-quasi-ordering for hereditary classes with finite distinguishing number. 

We also show that the decision procedure of Atminas, Collins, Foniok and Lozin to decide the Bell number (and which now also decides well-quasi-ordering for classes of finite distinguishing number) has run time bounded by an explicit (quadruple exponential) function of the order of the largest minimal forbidden induced subgraph of the class.

\end{abstract}

\section{Introduction}

Whereas the question of whether a downset is well-quasi-ordered or not is answered completely for the graph minor ordering (due to Robertson and Seymour~\cite{robertson:graph-minors-i-xx:}) and for the subgraph ordering (due to Ding~\cite{ding:subgraphs-and-w:}), for the induced subgraph ordering the question remains largely open, as the interface between well-quasi-ordering and not seems highly complex. Nevertheless, of the three orderings, classes of graphs closed under taking induced subgraphs are the only ones that can contain arbitrarily large graphs with high edge density, which makes this particular ordering indispensable.

A \emph{hereditary property}, or (throughout this paper) \emph{class} of graphs, is a collection of graphs that is closed under the induced subgraph ordering. That is, if $\C$ is a class, $G\in \C$ and $H$ is an induced subgraph of $G$, then $H$ must also lie in $\C$. Many natural collections of graphs form classes, for example perfect graphs and chordal graphs are both hereditary properties of graphs.

A common -- and computationally convenient -- way to describe graph classes is via their set of \emph{minimal forbidden graphs}. While some classes (such as perfect graphs and chordal graphs) have infinitely many minimal forbidden graphs, for others (such as cographs) this set is finite, in which case we say that the class is \emph{finitely defined}. For brevity, since we will always be dealing with minimal forbidden graphs, we will henceforth drop the word `minimal' and instead simply refer to `forbidden graphs.' 

The collection of forbidden graphs of a class is an antichain in the induced subgraph ordering (that is, no graph is an induced subgraph of another). Since there are non-finitely-defined classes, we see that the set of all finite graphs with the induced subgraph relation is not a well-quasi-order. However, we may still ask whether a class (having at least one forbidden graph) is \emph{well-quasi-ordered}, i.e.\ (since all classes are well-founded) if it contains no infinite antichains. Well-quasi-ordering in graph classes has received considerable attention, see, for example~\cite{daligault:well-quasi-orde:,damaschke:induced-subgrap:,ding:on-canonical-antichains:,korpelainen:two-forbidden:}. 

Classes that are well-quasi-ordered often possess a stronger property: well-quasi-ordering by the \emph{labelled} induced subgraph relation. Formally speaking, let $(W,<)$ be a quasi-order of labels, and consider graphs whose vertices are labelled by elements of $W$. We say that $H$ is a \emph{labelled induced subgraph} of $G$ if there is an isomorphism from $H$ to an induced subgraph of $G$ in which  each $v\in V(H)$ is mapped to some $u\in V(G)$ such that the label of $v$ is less than or equal to the label of $u$ in the ordering on $W$. A class of (unlabelled) graphs $\C$ is \emph{labelled well-quasi-ordered} if whenever $(W,<)$ is a well-quasi-order, the collection of labelled graphs from $\C$ contains no infinite antichain in the labelled induced subgraph ordering. 

In recent years, labelled well-quasi-ordering has emerged over unlabelled well-quasi-ordering in significance. For example Daligault, Rao and Thomass\'e~\cite{daligault:well-quasi-orde:} conjectured that labelled well-quasi-ordered classes must have bounded `clique-width'. This conjecture remains open (and is still widely believed), but it was shown recently~\cite{lozin:wqo-vs-cw} that the corresponding question for unlabelled well-quasi-ordering is false.

Given a graph class $\C$, the \emph{speed} of $\C$ is the sequence $(|\C^n|)_{n\geq1}$, where $\C^n$ denotes the collection of graphs in $\C$ on the vertex set $\{1,2,\dots,n\}$. Independently, Scheinerman and Zito~\cite{scheinerman:on-the-size:} and Alekseev~\cite{alekseev:range-of-values:}, observed that there were a number of discrete layers in the range of possible speeds. From `slowest' to `fastest', these are constant, polynomial, exponential, factorial and superfactorial.

Balogh Bollob\'as and Weinreich~\cite{balogh:the-speed-of-he:} showed that the factorial layer (in which $|\C^n| = n^{\theta(n)}$) can further be divided into two by the \emph{Bell numbers}, $\B_n$, which count the number of partitions of a set of size $n$. We say that a class $\C$ is \emph{above the Bell numbers} if there exists $N\geq 1$ such that $|\C^n| \geq \B_n$ for all $n\geq N$, and below otherwise. In~\cite{balogh:the-speed-of-he:} it was shown that all classes below the Bell numbers are of constant, polynomial or exponential speed, or there exists $k$ such that $|\C^n| = n^{(1-1/k+o(1))n}$. By contrast, all classes above the Bell numbers have speeds $|\C^n|\geq n^{(1+o(1))n}$.

In studying this dichotomy at the Bell numbers, the authors of~\cite{balogh:a-jump-to-the-bell:} developed a concept which we will call the `distinguishing number'. We postpone the details until later, but every class below the Bell numbers has finite distinguishing number, while there are 13 minimal classes above the Bell numbers that have infinite distinguishing number. Building on this, Atminas, Collins, Foniok and Lozin~\cite{atminas:deciding-the-bell:} describe a procedure to determine whether a finitely defined class lies above or below the Bell numbers. However, although the procedure is guaranteed to halt on any finite input, neither the theory underpinning the procedure in~\cite{atminas:deciding-the-bell:} nor the extremal arguments in~\cite{balogh:a-jump-to-the-bell:} can be used directly to obtain explicit bounds on the running time of the algorithm. 

Our main result comprises two distinct parts, and is as follows.

\begin{theorem}\label{thm-section-one-statement}
Let $\C$ be a class of graphs with finite distinguishing number. Then $\C$ is labelled well-quasi-ordered if and only if $\C$ is below the Bell numbers. Moreover, if $\C$ is finitely defined and above the Bell numbers, then it is not well-quasi-ordered. 

The running time $r=r(m)$ of the procedure to decide whether $\C$ lies above or below the Bell numbers satisfies $\log\log\log\log \left(r(m)\right) \leq m^{1+o(1)}$, where $m$ is the order of the largest forbidden graph of $\C$.
\end{theorem}

Note that one half of the first part of Theorem~\ref{thm-section-one-statement} (that classes below the Bell numbers are labelled well-quasi-ordered) is readily deduced from the literature, by combining a structural result due to Balogh, Bollob\'as and Weinreich~\cite[Theorem 30]{balogh:the-speed-of-he:} with a result relating this structure to well-quasi-ordering in Korpelainen and Lozin~\cite[Theorem 2]{korpelainen:two-forbidden:}. For completeness, we have included a direct proof of this direction in Appendix~\ref{subsec-belowbell}.

In proving Theorem~\ref{thm-section-one-statement}, we establish a number of other results. First, Theorem~\ref{thm:min-boundary} is a general result that shows `minimal' antichains (equating to the `minimal bad sequences' of Nash-Williams~\cite{nash-williams:on-well-quasi-o:}) and `boundary classes' (see Korpelainen and Lozin~\cite{korpelainen:boundary-properties:}) refer to the same objects. Second, Theorem~\ref{thm:pwh-boundary} identifies infinitely many minimal antichains of graphs (or, equivalently, boundary classes), lying (minimally) above the Bell numbers. Finally, in Theorem~\ref{periodic} we adapt the extremal arguments used by Balogh, Bollob\'as and Weinreich in~\cite{balogh:the-speed-of-he:,balogh:a-jump-to-the-bell:} in order to obtain explicit bounds on the run time of the existing procedure from~\cite{atminas:deciding-the-bell:}. Specifically, our Theorem~\ref{thm:infinitetofinite} can be thought of as a `finite' version of Theorem~20 from~\cite{balogh:a-jump-to-the-bell:} (presented as Theorem~\ref{thm:infinite} in this work), and the approach we take to complete the proof is comparable, although rather more intricate.

The rest of this paper is organised as follows. In Section~\ref{preliminaries} we define the distinguishing number and the characterisation of classes with infinite distinguishing number. Section~\ref{sec-antichains} contains the (general and self-contained) result that relates minimal antichains to boundary classes. Section~\ref{5.5} establishes that the boundary for well-quasi-ordering coincides with the Bell numbers. Section~\ref{sec-bound} contains the required (largely Ramsey-theoretic) arguments to provide the bounds on the decision procedure. Finally, we present a number of concluding remarks and future directions in Section~\ref{sec-conclusion}.

\section{Distinguishing number and the Bell dichotomy}\label{preliminaries}

For a graph $G$ and a set $X=\{v_1, \dotsc, v_t\} \subseteq V(G)$,
we say that the pairwise disjoint subsets $U_1$,~\dots,~$U_m$ of~$V(G)$ are
\emph{distinguished} by~$X$ (or $X$ \emph{distinguishes} $U_1,\dots, U_m$) if for each~$i$, all vertices of~$U_i$
have the same neighbourhood in $X$, and for each $i \neq j$,
vertices $x \in U_i$ and $y \in U_j$ have different neighbourhoods in~$X$. (Note that we do not require that $X$ is disjoint from the sets $U_i$.)

Given a graph $G$, clearly, for any set $X\subseteq V(G)$ one can distinguish at most $2^{|X|}$ sets, but not all of these necessarily contain a large number of vertices. In order to establish the ``jump'' from classes whose speeds are below the Bell numbers to those whose speeds are above, the authors of \cite{balogh:the-speed-of-he:,balogh:a-jump-to-the-bell:} identify a parameter for graph classes which, roughly speaking, identifies whether the class contains graphs with arbitrarily many sets each of arbitrary cardinality that are distinguished by some $X\subseteq V(G)$.

More precisely, the \emph{distinguishing number}\footnote{Note that the term `distinguishing number' is commonly used for a graph parameter that relates to vertex labellings preserved under automorphisms. We will not need this notion here, so have kept the terminology used in earlier articles, such as~\cite{atminas:deciding-the-bell:}.} of a class $\C$, denoted $k_\C$, is defined as follows:
\begin{compactenum}[(a)]
\item
If for all $k, \ell \in \mathbb{N}$ we can find a graph $G \in\C$ that admits
some $X \subset V(G)$ distinguishing at least $\ell$~sets, each  of size at least~$k$,
then we say that the distinguishing number is infinite, and write $k_{\C}=\infty$.
\item
Otherwise, there exists a pair $(k, \ell)$ such that for any graph $G\in \C$, any subset $X\subset V(G)$ distinguishes at most $\ell$~sets each of size at least~$k$.
We define $k_{\C}$ to be the minimum value of~$k$ in all such pairs.
\end{compactenum}

When $k_\C < \infty$, we set $\ell_\C$ to be the largest $\ell$ such that for every $k$ there exists $G\in\C$ and a set $X\subseteq V(G)$ which distinguishes $\ell$ sets, each of size at least $k$. In other words, $\ell_\C$ measures how many arbitrarily large sets can be distinguished in graphs in $\C$.

\paragraph{Infinite distinguishing number}
When a class $\C$ has infinite distinguishing number ($k_\C=\infty$), then \cite[Theorem 20]{balogh:a-jump-to-the-bell:} shows that $|\C_n|\geq \mathcal{B}_n$, the $n$th Bell number. In fact, there are precisely 13 minimal classes with $k_\C=\infty$, illustrated in Figure~\ref{fig:uni}: 
\begin{description}
\item[$\X_1$:] every graph is a disjoint union of cliques.
\item[$\X_2,\X_3,\X_4,\X_5$:] every graph can be partitioned into two sets $V_1$ and $V_2$, such that every $v\in V_2$ has at most one neighbour in $V_1$. For $\X_2$, $V_1$ and $V_2$ are both independent sets; in $\X_3$, $V_1$ is an independent set and $V_2$ a clique; $\X_4$ has $V_1$ a clique and $V_2$ an independent set; and in $\X_5$ both $V_1$ and $V_2$ are cliques.

\item[$\X_6,\X_7$:] every graph can be partitioned into two sets $V_1$ and $V_2$, such that the vertices in $V_1$ can be linearly ordered by the inclusion of their neighbourhoods in $V_2$. For $\X_6$, $V_1$ and $V_2$ are both independent sets; for $\X_7$, $V_1$ is an independent set and $V_2$ is a clique.
\item[$\X_8,\dots,\X_{13}$:] The classes of the complements of the graphs in $\X_1,\dots,\X_6$, respectively. (Note that $\X_7=\overline{\X_7}$.)
\end{description}

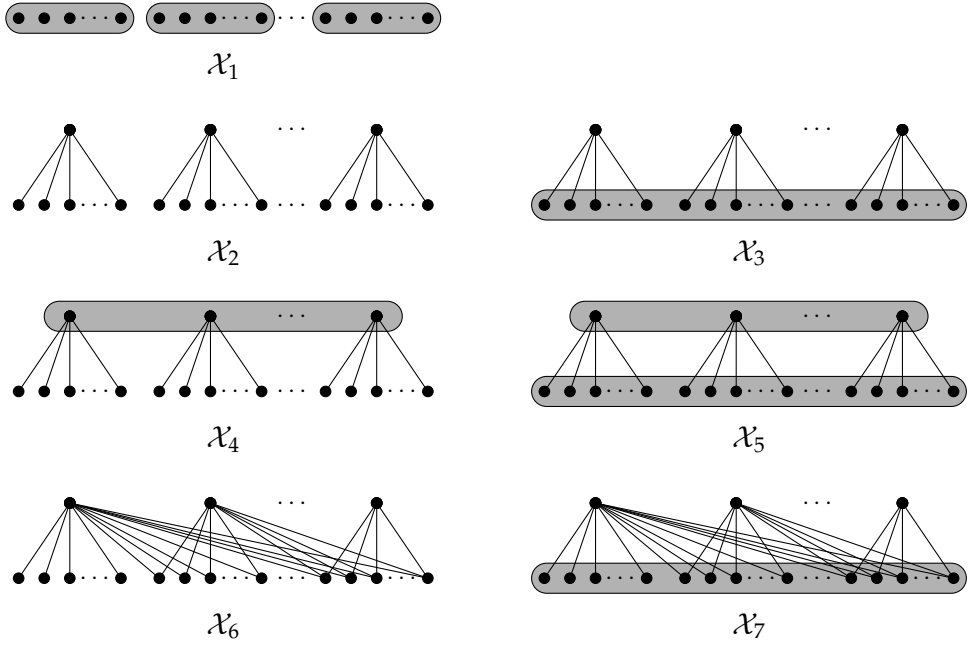
\begin{figure}[t]
\begin{center}
\begin{tabular}{ccc}
\begin{tikzpicture}[xscale=1.7]
\foreach \x in {0,1.1,2.4}
{
	\draw[fill=black!30, rounded corners=2mm] (\x-0.5,-0.2) rectangle (\x+0.5,0.2);
	\foreach \xx in {-0.4,-0.2,0,0.4}
		\node at (\x+\xx,0) {};
	\node[draw=none,fill=none] at (\x+0.21,0) {\footnotesize $\cdots$};		
}
\node[draw=none,fill=none] at (1.75,0) {\footnotesize $\cdots$};	
\end{tikzpicture}
\\
$\X_1$\\[10pt]
\begin{tikzpicture}[xscale=1.7]
\foreach \x in {0,1.1,2.4}
{
	\foreach \xx in {-0.4,-0.2,0,0.4}
		\draw (\x,1) node {} -- (\x+\xx,0) node {};
	\node[draw=none,fill=none] at (\x+0.21,0) {\footnotesize $\cdots$};		
}
\node[draw=none,fill=none] at (1.75,0) {\footnotesize $\cdots$};	
\node[draw=none,fill=none] at (1.75,1) {\footnotesize $\cdots$};	
\end{tikzpicture}
&\rule{10pt}{0pt}&
\begin{tikzpicture}[xscale=1.7]
\draw[fill=black!30, rounded corners=2mm] (-0.5,-0.2) rectangle (2.9,0.2);
\foreach \x in {0,1.1,2.4}
{
	\foreach \xx in {-0.4,-0.2,0,0.4}
		\draw (\x,1) node {} -- (\x+\xx,0) node {};
	\node[draw=none,fill=none] at (\x+0.21,0) {\footnotesize $\cdots$};		
}
\node[draw=none,fill=none] at (1.75,0) {\footnotesize $\cdots$};	
\node[draw=none,fill=none] at (1.75,1) {\footnotesize $\cdots$};	
\end{tikzpicture}
\\
$\X_2$&&$\X_3$\\[10pt]
\begin{tikzpicture}[xscale=1.7]
\draw[fill=black!30, rounded corners=2mm] (-0.2,0.8) rectangle (2.6,1.2);
\foreach \x in {0,1.1,2.4}
{
	\foreach \xx in {-0.4,-0.2,0,0.4}
		\draw (\x,1) node {} -- (\x+\xx,0) node {};
	\node[draw=none,fill=none] at (\x+0.21,0) {\footnotesize $\cdots$};		
}
\node[draw=none,fill=none] at (1.75,0) {\footnotesize $\cdots$};	
\node[draw=none,fill=none] at (1.75,1) {\footnotesize $\cdots$};	
\end{tikzpicture}
&\rule{10pt}{0pt}&
\begin{tikzpicture}[xscale=1.7]
\draw[fill=black!30, rounded corners=2mm] (-0.2,0.8) rectangle (2.6,1.2);
\draw[fill=black!30, rounded corners=2mm] (-0.5,-0.2) rectangle (2.9,0.2);
\foreach \x in {0,1.1,2.4}
{
	\foreach \xx in {-0.4,-0.2,0,0.4}
		\draw (\x,1) node {} -- (\x+\xx,0) node {};
	\node[draw=none,fill=none] at (\x+0.21,0) {\footnotesize $\cdots$};		
}
\node[draw=none,fill=none] at (1.75,0) {\footnotesize $\cdots$};	
\node[draw=none,fill=none] at (1.75,1) {\footnotesize $\cdots$};	
\end{tikzpicture}
\\
$\X_4$&&$\X_5$\\[10pt]
\begin{tikzpicture}[xscale=1.7]
\foreach \x in {0,1.1,2.4}
{
	\foreach \xx in {-0.4,-0.2,0,0.4}
		\draw (\x,1) node {} -- (\x+\xx,0) node {};
	\node[draw=none,fill=none] at (\x+0.21,0) {\footnotesize $\cdots$};		
}
\foreach \xx in {-0.4,-0.2,0,0.4}
{
	\draw (0,1) -- (1.1+\xx,0);
	\draw (0,1) -- (2.4+\xx,0);
	\draw (1.1,1) -- (2.4+\xx,0);
}
\node[draw=none,fill=none] at (1.75,0) {\footnotesize $\cdots$};	
\node[draw=none,fill=none] at (1.75,1) {\footnotesize $\cdots$};	
\end{tikzpicture}
&\rule{10pt}{0pt}&
\begin{tikzpicture}[xscale=1.7]
\draw[fill=black!30, rounded corners=2mm] (-0.5,-0.2) rectangle (2.9,0.2);
\foreach \x in {0,1.1,2.4}
{
	\foreach \xx in {-0.4,-0.2,0,0.4}
		\draw (\x,1) node {} -- (\x+\xx,0) node {};
	\node[draw=none,fill=none] at (\x+0.21,0) {\footnotesize $\cdots$};		
}
\node[draw=none,fill=none] at (1.75,0) {\footnotesize $\cdots$};	
\node[draw=none,fill=none] at (1.75,1) {\footnotesize $\cdots$};	
\foreach \xx in {-0.4,-0.2,0,0.4}
{
	\draw (0,1) -- (1.1+\xx,0);
	\draw (0,1) -- (2.4+\xx,0);
	\draw (1.1,1) -- (2.4+\xx,0);
}
\end{tikzpicture}
\\
$\X_6$&&$\X_7$\\
\end{tabular}
\end{center}
\caption{Typical representatives from the minimal classes $\X_1,\dots,\X_7$ with infinite distinguishing number (grey regions indicate cliques). The other six minimal classes are formed by complementation, $\X_8=\overline{\X_1}, \dots, \X_{13}=\overline{\X_6}$ (note that $\X_7 = \overline{\X_7}$).}
\label{fig:uni}
\end{figure}

\begin{theorem}[{Balogh, Bollob\'as and Weinreich~\cite[Theorem 20]{balogh:a-jump-to-the-bell:}}]
\label{thm:infinite}
Let $\C$ be a class with $k_\C=\infty$.
Then $\C$ contains one or more of the (minimal) classes of graphs $\X_1,\X_2,\dots,\X_{13}$.
\end{theorem}

Atminas, Collins, Foniok and Lozin~\cite{atminas:deciding-the-bell:} show that one can readily determine (in polynomial time) whether a given finitely defined graph class $\C=\free(G_1,\dots,G_s)$ contains any of the classes $\X_1,\dots,\X_{13}$. All such classes lie (minimally) above the Bell numbers, so in order to establish whether a finitely defined class lies above or below the Bell numbers we now restrict our attention to classes with finite distinguishing number.

\paragraph{Finite distinguishing number}
For classes with $k_\C<\infty$, the distinction between classes below and above the Bell numbers is more subtle. Building on a basic structural description (which we will introduce in Section~\ref{5.5}) of such classes in \cite{balogh:a-jump-to-the-bell:}, Atminas et al~\cite{atminas:deciding-the-bell:} show that the minimal classes above the Bell numbers with $k_\C<\infty$ are characterised by $\P(w,H)$ classes, which we now review.

A \emph{word} $w$ over an alphabet $A$ is a (possibly infinite) sequence $w=w_1w_2w_3\cdots $ of letters $w_i\in A$. The \emph{length}, $|w|$ of a word is the number of symbols it contains. We say that a finite word $u=u_1u_2\cdots u_m$ is a \emph{factor} of $w=w_1w_2\cdots $ if there exists an index $a$ such that $u_i = w_{a+i}$ for all $1\leq i\leq m$. In the case where $u$ is a factor of $w$ with index $a=0$, we say that $u$ is an \emph{initial segment} of $w$.


Fix a finite alphabet $A$, a (possibly infinite) word $w$ over $A$, and let $H$ be a graph with loops allowed with vertex set $V(H)=A$.
For any increasing sequence $i_1 < i_2 < \dotsb <i_k$ of positive
integers (with $i_k\le|w|$ if $w$ has finite length), define $G_{w,H}(i_1, i_2, \ldots, i_k)$ to be the graph
with vertex set $\{i_1, i_2, \dotsc, i_k\}$ and an edge between
$i_j$ and $i_{j'}$ if and only if either
\begin{compactitem}
\item $|i_j-i_{j'}|=1$ and $w_{i_j}w_{i_{j'}} \notin{E(H)}$, or
\item $|i_j-i_{j'}| > 1$ and $w_{i_j}w_{i_{j'}} \in E(H)$.
\end{compactitem}


Define $\P(w, H)$ to be the hereditary class consisting of the
graphs $G_{w, H}(i_1, i_2, \ldots, i_k)$ for all finite increasing
sequences $i_1 < i_2 < \dotsb < i_k$ of positive integers.

The classes $\P(w,H)$ enable a characterisation of the minimal classes with finite distinguishing number that lie above the Bell numbers. In this characterisation, we can restrict our attention to infinite words that satisfy a certain form of periodicity. Specifically, we say that a word $w$ is \emph{almost periodic} if for any finite factor $u$ of $w$, there is a constant $c$ such that every factor of $w$ of length $c$ contains $u$ as a factor.

\begin{theorem}[Atminas et al~{\cite[Theorems 3.10 and 3.13]{atminas:deciding-the-bell:}}]\label{factors}
Suppose $\C$ is a hereditary class above the Bell numbers with finite distinguishing number $k_\C$.
Then $\C$ contains some $\P(w, H)$ where $H$ is a graph (with loops allowed) of order at most $\ell_\C$, and $w$ is an infinite almost periodic word over the alphabet $V(H)$.

Furthermore, each such class $\P(w,H)$ is minimally above the Bell numbers (i.e.\ any proper hereditary subclass of $\P(w,H)$ lies below the Bell numbers).
\end{theorem}

In the case of finitely defined classes, (i.e.\ classes of the form $\C=\free(G_1,\dots,G_s)$), one can go further: $\C$ must contain some class $\P(w,H)$ where $w$ is \emph{periodic}, i.e.\ there exists $c$ such that $w_i = w_{i+c}$ for all $i\geq 1$, see~\cite[Theorem 4.6]{atminas:deciding-the-bell:}. 

This characterisation readily leads to a procedure to determine whether a given finitely defined class $\C=\free(G_1,\dots,G_s)$ is above or below the Bell numbers -- see Algorithm 4.7 of~\cite{atminas:deciding-the-bell:}. However, although this algorithm (which takes as input the forbidden  graphs $G_1,\dots,G_s$) will necessarily eventually terminate, the underpinning theory gives no bound for $\ell_\C$, and thus in the condition `$\C$ must contain some class $\P(w,H)$' there is no bound for the number of vertices in the graph $H$. We will provide such a bound later in this paper.

\section{Minimal antichains and boundary classes}\label{sec-antichains}

In this short section, we will establish some generic theory concerning infinite antichains which will be useful when we consider classes above the Bell numbers in the next section. In particular, we establish that the concept of a \emph{boundary class}, defined in~\cite{korpelainen:boundary-properties:}, is equivalent to the concept of a \emph{minimal} infinite antichain, first introduced by Nash-Williams~\cite{nash-williams:on-well-quasi-o:}. Except for basic definitions from the introduction, this section is self-contained.

Let $\A$ and $\B$ be infinite antichains of graphs, under the induced subgraph ordering\footnote{Although we specialise to graphs here, much of the theory we present here holds more generally for generic posets.}. We say that $\A\preceq \B$ if for every $B\in \B$ there exists $A\in \A$ such that $A$ is an induced subgraph of $B$. An infinite antichain $\A$ is \emph{minimal} if it is minimal under $\preceq$. 

By following the ``minimal bad sequence'' argument given by Nash-Williams~\cite{nash-williams:on-well-quasi-o:}, we can obtain the following result (for an explicit proof, see Gustedt~\cite{gustedt:finiteness-theo:}).

\begin{proposition}Every hereditary class $\C$ that is not well-quasi-ordered by the induced subgraph relation contains a minimal antichain.\end{proposition}

Given an infinite antichain $\A$, define the \emph{proper closure} to be the hereditary property of all graphs that are (proper) induced subgraphs of elements of $\A$, i.e.\ $\P^{<\A} = \{G: G < A \text{ for some } A\in\A\}$. For $G\in\P^{<\A}$, let $\A^{\parallel G} = \{A\in\A: G\not\leq A\}$. Finally, we say that  $\A$ is \emph{complete}\footnote{Some authors use the term ``maximal'' instead of complete. We have avoided this to prevent confusion with the ordering $\preceq$.} if for every graph $G$, we have either $G\in\A$, or there exists $A\in\A$ such that $G<A$ or $A<G$. We have the following characterisation of minimal antichains.

\begin{proposition}[{Gustedt~\cite[Theorem~6]{gustedt:finiteness-theo:}}]\label{minimal-char}
Let $\A$ be an infinite antichain of graphs under the induced subgraph ordering. Then $\A$ is minimal if and only if $\P^{<\A}$ is well-quasi-ordered, $\A^{\parallel G}$ is finite for every $G\in\P^{<\A}$, and $\A$ is complete.
\end{proposition}

By virtue of the above result, the proper closure of a minimal antichain is equal to the hereditary class whose set of minimal forbidden elements is $\A$, that is, $\P^{<\A} = \free(\A)$. 

An alternative viewpoint to this notion of minimal antichain was provided by Korpelainen et al~\cite{korpelainen:boundary-properties:}. We say that a graph class $\C$ is a \emph{limit class} if there is a sequence of classes $\C_1\supset \C_2\supset \cdots$ for which each $\C_i$ is not well-quasi-ordered, and $\C = \cap_{i=1}^\infty \C_i$. Trivially, every infinitely defined class is a limit class. Restricting to finitely defined classes, this notion simply characterises the non-well-quasi-ordered classes, as the following lemma shows.

\begin{lemma}[{Korpelainen et al~\cite[Lemma 1]{korpelainen:boundary-properties:}}]\label{limit-class}
A finitely defined class is a limit class if and only if it is not well-quasi-ordered.
\end{lemma}

Note in particular that any class that fails to be well-quasi-ordered (finitely defined or otherwise) is a limit class. Next, we say that any class $\C$ is a \emph{boundary class} if it is a minimal limit class. By necessity, boundary classes must be well-quasi-ordered, and thus by the above are not finitely defined. We have the following characterisation.

\begin{lemma}[{Korpelainen et al~\cite[Lemma 5]{korpelainen:boundary-properties:}}]\label{criterion}
A limit class $\C = \free(\A)$ is a boundary class if and only if for every $G\in\C$ there is a finite subset $T\subset\A$ such that $\free(\{G\}\cup T)$ is well-quasi-ordered.\end{lemma}

We now show that boundary classes and minimal antichains describe the same structures, and thus we can use these concepts interchangeably.

\begin{theorem}\label{thm:min-boundary}
Let $\A$ be an antichain. Then $\A$ is minimal if and only if  $\free(\A)$ is a boundary class. 
\end{theorem}

\begin{proof}
Suppose first that $\A$ is minimal (which implies that $\A$ is an infinite antichain). Write $\A=\{A_1, A_2, \ldots\}$, and let $\C=\free(\A)$. Set $\C_i =\free(A_1,\dots,A_i)$ so that $\C = \cap_{i=1}^\infty\C_i$, which shows that $\C$ is a limit class. By Proposition~\ref{minimal-char}, $\C=\free(\A) = \P^{<\A}$ is well-quasi-ordered, and for any $G\in\C$ the set $\A^{\parallel G}$ is finite. Therefore we have that for any $G\in\C$, the hereditary class $\free(\{G\}\cup\A^{\parallel G})$ is finitely defined, and as it is a subclass of $\C=\P^{<\A}$, it is well-quasi-ordered. Thus, by Lemma~\ref{criterion}, $\C$ is a boundary class.

Now suppose that $\C=\free(\A)$ is a boundary class. First, note that $\C$ must be well-quasi-ordered. Otherwise, $\C$ would contain an infinite antichain $\B$, say. Pick some $B\in\B$, and consider $\C\cap\free(B)\subsetneq \C$: this class is not well-quasi-ordered (it contains $\B\setminus\{B\}$), and therefore it is a limit class, contradicting the minimality of $\C$. Thus $\C$ is well-quasi-ordered. By Lemma~\ref{limit-class}, this also implies that $\A$ is an infinite antichain.

We next show that $\A$ is complete. Suppose to the contrary that there exists $G$ such that $\{G\}\cup\A$ is an infinite antichain. Then $G\in\C$, and $\free(\{G\}\cup\A)$ is a proper subclass of $\C$, and since it is not finitely defined, it is a limit class, a contradiction showing that $\A$ is complete. Note that $\A$ being complete implies that $\free(\A) = \P^{<\A}$, and therefore $\P^{<\A}=\C$ is well-quasi-ordered.

Finally, for any $G\in\free(\A)=\P^{<\A}$, by Lemma~\ref{criterion}, we can find a finite set $T \subset \A$ such that $\free(\{G\} \cup T)$ is well-quasi-ordered. This means that the set $\A^{\parallel G}\setminus T$ (which is contained in $\free(\{G\} \cup T)$) must be finite, which implies that $\A^{\parallel G}$ is finite. Thus, by Proposition~\ref{minimal-char} $\A$ is a minimal antichain.
\end{proof}

%
%
%
%
%
%
%
%
%

%
%
%
%
%
%
%
%
\section{Characterising well-quasi-ordering}\label{5.5}

Our aim in this section is to characterise the boundary between well-quasi-ordering and not for classes with finite distinguishing number. 

We first briefly mention the case of classes that lie below the Bell numbers. As noted in the introduction, the following result is readily deduced by combining Theorem 30 of~\cite{balogh:the-speed-of-he:} with Theorem 2 of~\cite{korpelainen:two-forbidden:}.

\begin{theorem}[Balogh, Bollob\'as and Weinreich~{\cite{balogh:the-speed-of-he:}} \& Korpelainen and Lozin~{\cite{korpelainen:two-forbidden:}}]\label{belowwqo}
Every hereditary class below the Bell numbers is labelled-well-quasi-ordered. 
\end{theorem}

Since Theorem~\ref{belowwqo} has not previously been explicitly mentioned, and since its proof is spread over two papers that use different notation, for the sake of completeness we have presented a proof of it in Appendix~\ref{subsec-belowbell}. Both the case of classes below the Bell numbers and those above rely on a concept known as `sparsification', which the next subsection handles.

We make here one further remark. In \cite{daligault:well-quasi-orde:} the following observation was made.

\begin{lemma}[Daligault, Rao and Thomass\'e~\cite{daligault:well-quasi-orde:}]\label{lemma:lwqo-fd} If $\C$ is a labelled well-quasi-ordered class, then $\C$ is finitely defined.\end{lemma}

This gives the following interesting corollary of Theorem~\ref{belowwqo}.

\begin{corollary} \label{finitedef}
Every hereditary class below the Bell numbers is finitely defined.
\end{corollary}

\subsection{\texorpdfstring{$(\ell,d)$}{(l,d)}-graphs and sparsification}

For an arbitrary graph $G$, let $U,W\subseteq V(G)$ be two (not necessarily disjoint) sets of vertices. Let
\[\Delta(U,W)= \max \{ |N(u) \cap W|, |N(w) \cap U| : u \in U, w \in W\},\]
denote the size of the largest neighbourhood of some vertex from $U$ or $W$ in the other set. Note that $\Delta(U,U)$ records the maximum degree of a vertex in the graph induced on the vertex set $U$.

Similarly, letting $\overline{N}(u)=V(G) \backslash (N(u) \cup \{u\})$ denote the non-neighbourhood of $u$, define \[\overline{\Delta}(U,W)=\max\{|\overline{N}(u) \cap W|, |\overline{N}(w) \cap U|: w \in W, u \in U \},\]
which records the size of the largest non-neighbourhood of a vertex from $U$ or $W$ in the other set.

A partition $\pi=\{V_1,V_2,\dotsc,V_{\ell'}\}$ of~$V(G)$ is an \emph{$(\ell,d)$-partition}
if $\ell'\leq \ell$ and for each pair of integers $1\leq i,j \leq \ell'$ (not necessarily distinct) either $\Delta(V_i, V_j) \leq d$ or $\overline{\Delta}(V_i, V_j) \leq d$. We say that $G$ is an \emph{$(\ell, d)$-graph} if it admits some $(\ell,d)$-partition. We will often refer to the sets $V_i$ as \emph{bags}.

If in some $(\ell,d)$ partition $\pi=\{V_1,V_2,\dotsc,V_{\ell'}\}$ of~$V(G)$ we have $|V_i|\geq t$ for some $t$ and for all $i=1,\dots,\ell'$, then we call $\pi$ a \emph{$t$-strong} $(\ell,d)$ partition, and $G$ is a $t$-strong $(\ell,d)$ graph.

We are particularly interested in $(\ell,d)$ graphs where $d$ is `small', even when the bags in the partition are large. If $\Delta (V_i, V_j) \leq d$ (respectively, $\overline{\Delta}(V_i, V_j) \leq d$), we say that the pair $(V_i,V_j)$ is \emph{$d$-sparse} (resp. \emph{$d$-dense}). Note that in a $t$-strong partition where $t\geq 2d+1$, the terms $d$-dense and $d$-sparse are mutually exclusive. 

The significance of $(\ell,d)$-graphs for classes with finite distinguishing number is the following result, which guarantees that every $G$ in $\C$ is an $(\ell,d)$-graph after removing a bounded number of vertices. Recall that when $k_\C<\infty$, we defined $\ell_\C$ to be the largest number of sets, each of size at least $k_\C$, that can be distinguished in any graph in $\C$.

\begin{lemma}[{Atminas et al~\cite[Lemma 2.11]{atminas:deciding-the-bell:}}] \label{finitedistinguishingnumber}
Let $\C$ be a class with~$k_{\C}<\infty$, , and let $t\geq 0$ be fixed. Then there exist $\ell_\C$, $d_{\C}$ and $c_{\C} = c_\C(t)$
such that for all $G \in {\C}$, the graph~$G$ contains an induced subgraph~$G'$
such that $G'$~is a $t$-strong $(\ell_{\C}, d_{\C})$-graph and ${|V(G) \backslash V(G')| < c_{\C}}$.
\end{lemma}

Note that a version of this result can be found in~\cite{balogh:the-speed-of-he:}, but the version we are using here guarantees that the $(\ell,d)$ partition is $t$-strong, for an arbitrary positive integer $t$: this is achieved simply by allowing $c_\C$ to be a function of $t$.

\paragraph{Sparsification}
Given a graph $G$ with an $(\ell,d)$-partition $\pi=\{V_1,V_2,\dotsc,V_{\ell'}\}$, define $\phi(G,\pi)$ 
to be the graph obtained from~$G$ by replacing edges between $V_i$ and $V_j$ with non-edges, and vice-versa, for every $d$-dense pair $(V_i,V_j)$. i.e.\ take the bipartite complement between the sets $V_i$ and $V_j$. We call $\phi$ the \emph{sparsification} of $G$ with respect to the partition $\pi$.

It was shown in \cite{atminas:deciding-the-bell:} that if $t\geq 5\times 2^\ell d$, then for a $t$-strong $(\ell,d)$-graph $G$ and any two $t$-strong $(\ell,d)$-partitions $\pi$ and $\pi'$, 
the graphs $\phi(G,\pi)$ and $\phi(G, \pi')$ are identical. In this case, therefore, we can simply talk about the \emph{sparsification} of $G$, without reference to any particular partition. This uniqueness is critical for the rest of our analysis.

%
%
%
%
%
%
%
%
%
%
\subsection{Above the Bell numbers}\label{subsec-abovebell}

In this subsection, by considering the minimal classes $\P(w,H)$ above the Bell numbers, we will prove the following.

\begin{theorem}\label{thm:abovebell}
Let $\C$ be a hereditary class with $k_\C<\infty$ which lies above the Bell numbers. Then $\C$ is not labelled well-quasi-ordered. Moreover, if $\C$ is finitely defined, then $\C$ is not well-quasi-ordered.
\end{theorem}

The starting point is Theorem~\ref{factors}, which guarantees that any such class contains some class $\P(w,H)$ where $w$ is almost periodic, and $H$ has at most $\ell_\C$ vertices. Our proof will show that the classes $\P(w,H)$ are boundary classes for well-quasi-ordering. In order to do this, we will show that we can use the method of sparsification to prove that some (specific) infinite sets of $(\ell,d)$-graphs form infinite antichains, by relating how one such graph can embed in another to embeddings in their sparsifications. We begin by reviewing a few preliminary results about $(\ell,d)$-graphs.

\begin{lemma} [{Balogh et al~\cite[Lemma 10]{balogh:a-jump-to-the-bell:}}]
\label{symdif1}
Let $G$ be a graph with an $(\ell,d)$-partition $\pi$. If two vertices $u,v \in V(G)$ are in the same bag $V\in\pi$,
then the symmetric difference of their neighbourhoods $N(u) \ominus N(v)$ is of size at most $2\ell d$.
\end{lemma}

For $t\geq 2d+1$, given any $t$-strong $(\ell,d)$-partition $\pi = \{V_1, V_2, \ldots, V_{\ell'}\}$ we define an equivalence relation~$\sim$ 
on the bags by putting $V_i \sim V_j$ if and only if for each~$k$, either $V_k$~is $d$-dense with respect to both $V_i$ and~$V_j$,
or $V_k$~is $d$-sparse with respect to both $V_i$ and~$V_j$.
Let us call a partition $\pi$ \emph{prime} if all its $\sim$-equivalence classes are of size~1.
If the partition $\pi$ is not prime, let $p(\pi)$ be the partition consisting of unions of bags in the $\sim$-equivalence classes
for~$\pi$.

\begin{lemma}[{Atminas et al~\cite[Lemma 2.3]{atminas:deciding-the-bell:}}] \label{symdif0}
Let $t\geq 2d+1$ and consider any $t$-strong $(\ell,d)$-graph $G$ with $t$-strong partition $\pi$.
Then $p(\pi)$ is a $t$-strong $(\ell , \ell d)$-partition. 
\end{lemma}

\begin{lemma}[{Atminas et al~\cite[Lemma 2.5]{atminas:deciding-the-bell:}}]\label{symdif2}
Let $G$ be a graph with a $t$-strong $(\ell,d)$-partition $\pi$ with $t\geq 2d+1$.
If two vertices $u, v \in V(G)$ belong to different bags of the partition $p(\pi)$,
then the symmetric difference of their neighbourhoods $N(x) \ominus N(y)$ is of size at least $t-2d$.  
\end{lemma}

Let $G$ and $H$ be graphs. An \emph{embedding} of $G$ into $H$ is an injective function $f:V(G)\rightarrow V(H)$ such that $uv\in E(G)$ if and only if $f(u)f(v)\in E(H)$. That is, an embedding is an instance of a mapping of $G$ into $H$ that witnesses $G$ as an induced subgraph of $H$. 

We are now ready to prove the result concerning embeddings that we require.  Note the condition that the graphs are $(5\times 2^\ell d)$-strong, which ensures that sparsification is unique regardless of which $(\ell,d)$-partition we take.

\begin{lemma} \label{embedding}
Suppose $G$ and $H$ are $(5\times 2^\ell d)$-strong $(\ell, d)$-graphs with corresponding partitions $\pi_G$ and $\pi_H$, respectively. Assume further that the partitions
$p(\pi_G)=\{V_1,\ldots, V_{\ell_G}\}$ and $p(\pi_H)=\{W_1,\ldots W_{\ell_H}\}$ 
have the same number $\ell'=\ell_G=\ell_H$ of bags. Then for every embedding $f$ of $G$ into~$H$, the following holds:

\begin{compactenum}[(1)]
\item There is a permutation $\sigma_f$ on $\{1,\dots,\ell_G\}$ such that $f(V_i) \subseteq W_{\sigma_f(i)}$;
\item The pair $(V_i, V_j)$ is $d$-dense if and only if the pair $(W_{\sigma_f(i)}, W_{\sigma_f(j)})$ is $d$-dense;
\item $f$ is an embedding of $\phi(G)$ into $\phi(H)$, i.e.\ $uv \in E(\phi(G))$ if and only if $f(u)f(v) \in E(\phi(H))$.  
\end{compactenum}
\end{lemma}

\begin{proof}
Consider $G$ and $H$ as in the statement and an embedding $f : G \rightarrow H$.
For $i \in \{1,\dots,\ell_G\}$ put $S(i)=\{k : f(V_i) \cap W_k \neq \emptyset\}$. 

We claim that for $i \neq j$ we have $S(i) \cap S(j) = \emptyset$.
Suppose, for the sake of contradiction, that for some $i \neq j$
we can find $k\in S(i)\cap S(j)$.
Then there are two vertices $v_i \in V_i$, $v_j \in V_j$ such that $f(v_i), f(v_j) \in W_k$.
Now $p(\pi_H)$ is an $(\ell, \ell d)$-partition by Lemma~\ref{symdif0} and $f(v_i)$, $f(v_j)$ are two vertices in the same bag of $p(\pi_H)$.
Thus by Lemma~\ref{symdif1} we conclude that $|N(f(v_i)) \ominus N(f(v_j))| \leq 2\ell(\ell d)=2\ell^2d$.
As $f$~is an embedding, we have $f(N(v_i) \ominus N(v_j)) \subseteq N(f(v_i)) \ominus N(f(v_j))$;
hence $|N(v_i) \ominus N(v_j)| \leq 2\ell^2d$.
However, $v_i$ and $v_j$ belong to different bags of~$p(\pi_G)$, so by Lemma~\ref{symdif2}
we obtain that $|N(v_i) \ominus N(v_j)| \geq 5 \times 2^\ell d-2d$. This implies $5 \times 2^\ell d-2d \leq 2\ell^2d$ which is a contradiction.

So for all $i \neq j$ we have $S(i) \cap S(j) = \emptyset$.
This implies that \[\ell_H \ge |S(1) \cup \cdots \cup S({\ell_G})|=|S(1)|+\cdots+|S({\ell_G})| \ge \ell_G.\]
By assumption, $\ell_G=\ell_H$, so $|S(i)|=1$ for each~$i$.
Therefore there exists a (unique) permutation~$\sigma_f$ of~$[\ell_G]$ such that $f(V_i) \subseteq W_{\sigma_f(i)}$ for each~$i$,
which proves~(1).

It is not hard to see that $V_i$ is $d$-dense with respect to $V_j$
if and only if $W_{\sigma_f(i)}$ is $d$-dense with respect to $W_{\pi_f(j)}$.
Indeed, if one pair is $d$-dense and the other $d$-sparse,
then since $f(V_i) \subseteq W_{\sigma_f(i)}$ and $f(V_j) \subseteq W_{\sigma_f(j)}$
we have that $(V_i, V_j)$ is both $d$-sparse and $d$-dense. This implies that $|V_i| \leq 2d+1$, a contradiction which establishes~(2).

Finally, to show (3), consider any $u,v \in V(G)$, $u \in V_i$ and $v \in V_j$.
Then by definition $uv \in E(\phi(G))$ if and only if $uv \in E(G)$ and $(V_i, V_j)$ is $d$-sparse, or $uv \notin E(G)$ and $(V_i, V_j)$ $d$-dense.
But as $uv \in E(G)$ if and only if $f(u)f(v) \in E(H)$,
and $(V_i, V_j)$ is $d$-sparse ($d$-dense, respectively) if and only if
$(W_{\sigma_f(i)}, W_{\sigma_f(j)})$ is $d$-sparse ($d$-dense, respectively),
we conclude that the statement is equivalent to saying that $f(u)f(v) \in E(\phi(H))$.
\end{proof}

%

\paragraph{Minimal antichains above the Bell numbers}

We know that the classes $\P(w,H)$ where $w$ is almost periodic are the minimal classes with finite distinguishing number that lie above the Bell numbers. We now show that every such class $\P(w,H)$ is a boundary class. First, we find an infinite collection of forbidden graphs of any such class $\P(w,H)$.

\begin{lemma}\label{lem:pwh-infinite}
Every class $\P(w,H)$ with $w$ almost periodic is defined by infinitely many minimal forbidden elements.
\end{lemma}

\begin{proof}
Consider the word $a=w_1w_2 \cdots w_k$, with $k$ large enough to contain every letter of
$H$ at least $10\times 2^{|H|}$ times, to guarantee that sparsification of the graph defined by $a$ is unique, since all such graphs are $(|H|,2)$-graphs.
Since $w$ is almost periodic, $a$ appears in $w$ infinitely often so we can pick some other occurrence 
$w_{\ell+1}w_{\ell+2} \cdots w_{\ell+k}$ for some $\ell>k$. 
Let $b=w_{k+1}w_{k+2} \cdots w_{\ell}$ be the word between these two occurrences. 
Consider the graph $G_{\ell}=G_{w, H}(1, 2, \ldots, \ell)$ formed from the vertices of the prefix $ab$, and
let $G_{\ell}^+$ denote the graph obtained from $G_{\ell}$ by complementing the edges between the vertices corresponding to $w_1$ and $w_{\ell}$. 

We claim that either $G_{\ell}^+$ is itself a minimal forbidden induced subgraph for $\P(w, H)$, or it contains one of order at least $k$. Suppose first that $G_{\ell}^+ \in \P(w,H)$. Then
there is some $p \geq \ell$ such that $G_{\ell}^+ $ is an induced subgraph of $G_p$. Fix any embedding $f$ of $G_\ell^+$ into $G_p$. By Lemma~\ref{embedding} it follows that $f$ is an embedding of
$\phi (G_{\ell})$ into an induced subgraph of $\phi (G_p)$. It is clear that $\phi(G_{\ell})$ is a cycle on $\ell$ vertices, but $\phi(G_p)$ is a path on $p$ vertices, and hence $f$ cannot embed $\phi(G_\ell)$ into $\phi(G_p)$, a contradiction, and so $G_\ell^+\not\in \P(w,H)$.

Now, suppose that $G_{\ell}^+$ is not itself a minimal forbidden subgraph of $\P(w,H)$. Observe that $G_{\ell}^+ - i \in \P(w, H)$ for every vertex $i$ satisfying $1\leq i\leq k$: this follows since $w$ begins with a prefix of the form $aba$, and thus $G_{\ell}^+ - i \cong G_{w,H}(i+1, i+2 \ldots, \ell+i-1)$. Thus, any minimal forbidden subgraph that is contained in $G_{\ell}^+$ (and there must be at least one) must contain all the vertices $1,\dots,k$, and thus has order at least $k$. 

To construct infinitely many minimal forbidden induced subgraphs we proceed as follows. 
We let $k_0=k$, $\ell_0=\ell$ and $F_0=G_\ell^+$ as above. $F_0$ is either itself a minimal forbidden induced subgraph, or it contains one of order at least $k_0$. For $i>0$, set $k_i=|F_{i-1}|+1$. 
Then as above we construct $G_{\ell_i}^+$ and let $F_i = G_{\ell_i}^+$.
Each $F_i$ is either itself a minimal forbidden induced subgraph of $\P(w,H)$, or it contains a minimal forbidden induced subgraph with at least $k_i$ vertices. In either case, since by construction $|F_{i-1}| < k_i$, the minimal forbidden induced subgraph contained in $F_i$ contains strictly more vertices than all the previous ones, completing the proof.
\end{proof}

\begin{theorem}\label{thm:pwh-boundary}
Every class $\P(w,H)$ with $w$ almost periodic is a boundary class for well-quasi-ordering by the induced subgraph relation. 
\end{theorem}

\begin{proof}
Consider the set of minimal forbidden induced subgraphs $M$ for $\P(w,H)$, i.e.\ $\P(w,H)=\free(M)$. 
By Lemma~\ref{lem:pwh-infinite}, $M$ is infinite, which implies that $\P(w,H)$ is a limit class. By Lemma~\ref{criterion}, it suffices to show that for every $G\in\P(w,H)$, there exists a finite set $M''\subset M$ such that $\free(\{G\}\cup M'')$ is well-quasi-ordered.

To prove that $\P(w,H)$ is a minimal limit class, consider $G \in \P(w,H)$. Then let $M'=\{H_i \in M | H_i \supset G\}$ and
$M''=M-M'$. It is clear that $\{G\} \cup M''$ is an antichain and hence a list of minimal induced subgraphs for the class
$\free(\{G\} \cup M'')=\free(\{G\} \cup M)$. As $\free(\{G\} \cup M)$ is a proper subclass of $\P(w,H)$, we know that it is below the Bell number by the second part of Theorem~\ref{factors}. Hence this class is well-quasi-ordered by Theorem~\ref{belowwqo} and finitely defined by Corollary~\ref{finitedef}.
Hence for $G \in \P(w, H)$ we found a finite set $M'' \subset M$ such that $\free(G \cup M'')$ is well-quasi-ordered, as required. 
\end{proof}

\begin{proof}[Proof of Theorem~\ref{thm:abovebell}]
First, if $\C$ is not finitely defined, then by Lemma~\ref{lemma:lwqo-fd}, $\C$ is not labelled well-quasi-ordered. So now we can assume that $\C=\free(G_1,\dots,G_s)$ is finitely defined.

By Theorem~\ref{factors}, $\C$ contains a class $\P(w,H)$ for some graph $H$ and almost periodic word $w$. By Theorem~\ref{thm:pwh-boundary}, $\P(w,H)$ is a boundary class, so by Theorem~\ref{thm:min-boundary}, $\P(w,H) = \free(\A)$ for a minimal antichain $\A$. Thus, if $\C = \free(G_1,\dots,G_s)$ is finitely defined, then each $G_i$ cannot lie in  $\P(w,H)$, and hence is either isomorphic to some graph in $\A$, or contains (at least one) graph  from $\A$. Either way, we conclude that all but finitely many of the graphs in $\A$ must also be contained in $\C$, and hence $\C$ is not well-quasi-ordered (and therefore also not labelled well-quasi-ordered).\end{proof}

Note that if $\C$ is not finitely defined then $\C$ immediately fails to be labelled well-quasi-ordered by Lemma~\ref{lemma:lwqo-fd}. However, such classes can still be well-quasi-ordered: for example, if $\C=\P(w,H)$ for some $H$ and almost periodic $w$.

%
%
%
%
%
%
%
%
%
\section{A bound on \texorpdfstring{$k_\C$}{k\_C} and \texorpdfstring{$\ell_\C$}{l\_C}}\label{sec-bound}

While Balogh et al~\cite{balogh:a-jump-to-the-bell:} provide an easy-to-check characterisation to determine whether a class $\C$ has finite or infinite distinguishing number, in this section we provide an upper bound on the value of $k_\C$ whenever $\C=\free(G_1, \ldots, G_s)$ has $k_\C<\infty$, as a function of the number of vertices in the largest graph in $\{G_1,\dots,G_s\}$. This also gives us a bound on the parameter $\ell_\C$, which counts how many sets can be distinguished in graphs in $\C$ which can have arbitrarily large cardinality. 

In order to establish this bound, we seek to derive a `finite' version of Theorem~\ref{thm:infinite}. First, we recall a number of concepts which are essentially in Balogh et al~\cite{balogh:a-jump-to-the-bell:}, although note that in~\cite{balogh:a-jump-to-the-bell:} these are defined in terms of hypergraphs, but our treatment here is equivalent.

Given a pair of disjoint sets of vertices $U$ and $V$ in a graph $G$, we say that $U$ and $V$ are \emph{joined} if $uv\in E(G)$ for all $u\in U$, $v\in V$,  and \emph{co-joined} if $uv\notin E(G)$ for all $u\in U$, $v\in V$.
Suppose $G$ is a graph with a partition $V(G)=X \cup V_1 \cup \cdots \cup V_r$ such that $X=\{v_1,\dots,v_r\}$, and $V_1,\dots, V_r$ is a collection of disjoint sets of vertices. 
Then we call $G$:

\begin{itemize}
\item an \emph{$r$-star} if each $v_i$ is joined to $V_i$ and co-joined to $V_j$ for $j \neq i$.
\item an \emph{$r$-costar} if each $v_i$ is co-joined to $V_i$ but joined to $V_j$ for $j \neq i$.
\item an \emph{$r$-skewchain} if each $v_i$ is joined to $V_j$ for $j \leq i$ and co-joined to $V_j$ for $j >i$.
\end{itemize}
Where necessary, we say that $v_i$ is the vertex of $X$ that \emph{corresponds} to the set $V_i$. In the case where $|V_i|\geq t$ for all $i$, we say that the $r$-star, $r$-costar or $r$-skewchain is \emph{$t$-strong}. In the case when $|V_i|=1$ for all $i$ and the graph 
$G$ is bipartite with one part equal to $X$, we call the graphs an \emph{$r$-matching}, an \emph{$r$-comatching} and an \emph{$r$-halfgraph}, respectively. 

\begin{proposition}[Balogh et al~\cite{balogh:a-jump-to-the-bell:}]\label{prop:f-well-defined}
For all $p,q,r \in \mathbb{N}$ there is a number $f(p,q,r)$ such that every bipartite graph in which one part
has $f(p,q,r)$ vertices with pairwise distinct neighbourhoods, contains a $p$-matching, a $q$-comatching, or an $r$-halfgraph as an induced subgraph.
\end{proposition} 

 We require two more results concerning unavoidable structures. First, for completeness we state Ramsey's theorem.

\begin{theorem}[Ramsey~\cite{ramsey}] There is a number $R(n)$ such that any graph with $R(n)$ vertices contains a clique or an independent set of size $n$.\end{theorem}

The second result we require concerns finding uniformity in graphs which can be partitioned into a large number of large parts. A suitable result is proved in~\cite{balogh:a-jump-to-the-bell:} (see Corollary~8), but the result we give here, which can be found in Brian Cook's MSc Thesis~\cite[Lemma 4.3.4]{cook:an-extension-of:}, provides us with an improved bound on the size of graph needing to be considered. 

\begin{lemma}[Cook~\cite{cook:an-extension-of:}]\label{lem:uniform-bags}
There is a number $n_k(t)$ such that if $G$ has a partition into $k$ sets $V_1,\dots,V_k$, each containing at least $n_k(t)$ vertices, then $G$ contains an induced subgraph $H$ with partition $W_1,\dots,W_k$ such that
\begin{enumerate}
\item[(i)] $|W_i| \geq t$ for all $i$,
\item[(ii)] $W_i\subseteq V_i$ for all $i=1,\dots,k$,
\item[(iii)] For each $i\neq j$, $W_i$ and $W_j$ are either joined or co-joined.
\end{enumerate}
\end{lemma}

\begin{proof}
We prove this by induction on $k$. Clearly, we can take $n_1(t) = t$.

Now consider a graph $G$ with a partition into $k+1$ sets $V_1,\dots, V_{k+1}$. First, if $V_1,\dots,V_k$ each contain at least $n_k(2t)$ vertices, then by induction we can find sets $W'_1,\dots,W'_k$ each of size $2t$ satisfying the statement of the lemma for the subgraph of $G$ induced on $V_1,\dots,V_k$.

Next, providing the set $V_{k+1}$ contains at least $t2^{2kt}$ vertices, then by the pigeonhole principle there is a set $W_{k+1}\subseteq V_{k+1}$ containing at least $t$ vertices which have the same neighbourhood with the $2kt$ vertices in $W'_1,\dots,W'_k$. 

Furthermore, since each set $W'_i$ contains $2t$ vertices, we can find a set $W_i\subseteq W'_i$ containing at least $t$ vertices which are either all adjacent or all nonadjacent to $W_{k+1}$. Therefore, $W_1,\dots,W_{k+1}$ induces the necessary subgraph $H$ of $G$, and we may take $n_{k+1}(t) = \max\{t2^{2kt},n_k(2t)\}$.
\end{proof}

For later reference, the final condition $n_{k+1}(t) = \max\{t2^{2kt},n_k(2t)\}$ together with $n_1(t)=t$ gives us that $n_2(t) = t2^{4t}$ and thence for $k\geq 3$ we have $n_k(t) = 2^{k-2}t2^{2^kt}$.

We are now nearly ready for our `finite' version of Theorem~\ref{thm:infinite}. Roughly speaking, given an arbitrary graph in which we can find a set $X$ that distinguishes a large number of large sets $V_1,V_2,\dots,V_{k}$, we seek an induced subgraph that is an $r$-star, $r$-costar or $r$-skewchain, and in which the relationship between and within the various sets of vertices is `uniform'. 

For an $r$-star, $r$-costar or $r$-skewchain with vertex partition $X\cup V_1\cup\cdots \cup V_r$ where $|X|=|V_i|=r$ for all $i$, we will take the set $X$ to be either an independent set or a clique, all $V_i$s will form independent sets or all $V_i$s form cliques, and all pairs $V_i$ and $V_j$ ($i\neq j$) will be joined or all pairs will be co-joined. This gives $3\times 2\times 2 \times 2=24$ minimal graphs, which we will call $U^r_1,U^r_2,\dots,U^r_{24}$, in some (arbitrary) order. Note that each $U^r_i$ is an $r$-star, an $r$-costar or an $r$-skewchain with $r^2+r$ vertices.

In the theorem below, we use $R_4(k)$ to denote the 4-coloured Ramsey number, sometimes denoted $R(k,k,k,k)$.

\begin{theorem} \label{thm:infinitetofinite}
If a graph $G$ contains a set $X$ which distinguishes $\ell(m)=f(M,M,M)$ sets where $M=R_4(2m-1)$, each of size $k(m)=n_M(R(m))$,
then the graph $G$ contains $U^m_i$ for some $1 \leq i \leq 24$. 
\end{theorem}

\begin{proof}
Let the collection of sets distinguished by $X$ be $\FF=\{V_i: 1 \leq i \leq \ell(m)\}$. Construct an auxiliary bipartite graph $H$ in 
which one part is the distinguishing set $X$ (with edges removed), and the other is $V=\{v_1,\dots,v_{\ell(m)}\}$ 
where $v_i$ has the same neighbourhood in $X$ as $V_i$ has with $X$ in $G$. 
Notice that, by definition of the distinguishing set $X$, it follows that no two vertices of $V$ have the same neighbourhood in $X$. 
Hence, part $V$ has $f(M,M,M)$ vertices with distinct neighbourhoods, which by Proposition~\ref{prop:f-well-defined} implies
that we can find a subset $X_1\subseteq X$ with $|X_1|=M$ and a set $\{v_{i_1},\dots,v_{i_{M}}\}$ such that $X_1\cup \{v_{i_1},\dots,v_{i_M}\}$ induces an $M$-matching, -comatching or -halfgraph inside $H$. In $G$, the induced subgraph on the vertices $X_1$ together with $\FF_1 = \{ V_{i_1},\dots,V_{i_M}\}\subseteq \FF$ is therefore a $k(m)$-strong $M$-star, -costar or -skewchain.

Next, as each of the $M$ sets $V_{i_j}\in\FF_1$ contains at least $k(m)=n_M(R(m))$ vertices, by Lemma~\ref{lem:uniform-bags} we can find sets $W_j\subseteq V_{i_j}$ with $|W_j|\geq R(m)$ so that between each pair $W_j$, $W_{j'}$ there are either all edges or none. Letting $\FF_2 = \{W_1,\dots,W_M\}$ and $X_2=X_1$, we have an $R(m)$-strong $M$-star, -costar or -skewchain where every pair of sets in $\FF_2$ is joined or co-joined.

Write $X_2=\{x_1,\dots,x_M\}$, with the indices arranged so that $x_i\in X_2$ corresponds to $W_i\in \FF_2$ for each $i$. Now consider the $M=R_4(2m-1)$ vertex-set pairs $(x_i,W_i)$. Between each $(x_i,W_i)$ and $(x_j,W_j)$ we identify an edge relation $e_{ij}$ with one of four colours, depending on whether $x_ix_j\in E(G)$ or not, and whether $W_i$ is joined or co-joined to $W_j$.

As there are $R_4(2m-1)$ pairs $(x_i,W_i)$, we can find $2m-1$ pairs so that all edges have the same colour. Let $\FF_3$ denote the collection of $2m-1$ sets $W_i$, and $X_3$ the corresponding $2m-1$ vertices $x_i$. Thus $X_3$ and $\FF_3$ induce in $G$ an $R(m)$-strong $(2m-1)$-star, -costar or -skewchain, in which $X_3$ is a clique or independent set, and either all edges are present between all pairs in $\FF_3$, or no edges are present between any pair.

Finally, as each set $U\in \FF_3$ contains $R(m)$ vertices, we can find a collection of $m$ vertices in $U$ which induce a clique or an independent set in $G$. Since $\FF_3$ contains $2m-1$ sets, at least $m$ of these sets in $\FF_3$ will induce cliques, or at least $m$ will induce independent sets. Thus by first reducing the sets in $\FF_3$ to subsets of size at least $m$, and then choosing at least $m$ of these sets, we can find a collection $\FF_4$ of $m$ sets each of size $m$, and corresponding vertices $X_4$ which creates an induced $U_i^m$ in $G$, for some $1\leq i\leq 24$.
\end{proof}

The astute reader will have noticed that although we have defined 24 `unavoidable' graphs, there are still only 13 minimal classes with infinite distinguishing number. The discrepancy here arises because of the unavoidable graphs $U_1^m,\dots,U_{24}^m$, one can find a disjoint union of $m$ cliques each of size at least $m$ in six of them, the Tur\'an graph $T(m^2,m)$ in another six, and the $m$-strong $m$-skewchain where one part is a clique and the other an independent set appears twice in the list. From this observation, we can readily deduce the following.

\begin{lemma}\label{lem:universal}Each graph $U_i^m$ is $m$-universal for some $\X_j$. That is, $U_i^m$ contains every graph on $m$ vertices from $\X_j$.\end{lemma}

We can now obtain an explicit bound for $k_\C$ and $\ell_\C$ for a finitely defined class $\C$ in which we know $k_\C<\infty$. 

\begin{theorem}
Let $\C=\free(G_1,G_2, \ldots, G_s)$ be a class with finite distinguishing number and let $m=\max \{|G_i|: 1 \leq i \leq s\}$ be the order of the largest forbidden graph. 
Then $k_\C \leq k(m)$ and $\ell_\C <\ell(m)$ where $k(m)$ and $\ell(m)$ are as given in Theorem~\ref{thm:infinitetofinite}. 
\end{theorem}

\begin{proof}
For the sake of contradiction, suppose that $\C$ contains a graph $G$ with a set $X$ distinguishing $\ell(m)$ sets each of size at least $k(m)$. 
Then, by Theorem~\ref{thm:infinitetofinite}, $G$ contains one of the graphs $U^m_i$ for some $i$. In particular, this means that $U^m_i\in\C$, and therefore by Lemma~\ref{lem:universal} every graph on $m$ vertices from some $\X_j$ is contained in $\C$. However, as $m$ is the size of the largest forbidden induced subgraph of $\C$, it follows that $\X_j\subseteq \C$, but this is impossible as it would imply $k_\C = \infty$.

From this, we see that for any graph $G\in\C$, one can distinguish at most $\ell(m)-1$ sets of size $k(m)$, and thus by definition $k_\C \leq k(m)$. Similarly, by definition of $\ell_\C$ we see that $\ell_\C < \ell(m)$.
\end{proof}


For the record, we now give a rough analysis of the sizes of $\ell(m)$ and $k(m)$. First note that Balogh et al~\cite{balogh:a-jump-to-the-bell:} proved that the function $f(p,q,r)$ in Proposition~\ref{prop:f-well-defined} satisfies
\[f(p,q,r) = 2(r-1)(r-2)f(p-1,q,r)f(p, q-1, r) + 1\]
whenever $p,q,r>2$, while $f(1,q,r)=f(p,1,r)=f(p,q,1)=1$, $f(2,q,r)=f(p,2,r)=r$ and $f(p,q,2)=2$. From this, we can deduce (e.g.\ by induction) that $f(p,q,r) \leq r^{2^{p+q}-3}\leq r^{2^{p+q}}$.

Next, naive bounds for Ramsey's theorem give us $R(m)\leq 2^{2m}$ and $M=R_4(2m-1)\leq 4^{4(2m-1)}\leq 2^{16m}$, while the comments after the proof of Lemma~\ref{lem:uniform-bags} tell us that $n_k(t)\leq 2^{k-2}t2^{2^kt}$. Thus, we have
\[\ell(m) = f(M,M,M) \leq M^{4^{M}}\leq 2^{16m4^{2^{16m}}}\]
and 
\[k(m) = n_M(R(m))\leq 2^{M-2}R(m)2^{2^{M}R(m)} \leq 2^{2^{2^{16m}}4^m}2^{2^{16m}}4^m.\]

\begin{theorem}\label{periodic}
Let $\C=\free(G_1,G_2, \ldots, G_s)$ be a class with finite distinguishing number and let $m=\max \{|G_i|: 1 \leq i \leq s\}$. 
If $\C$ is above the Bell numbers then it contains a class $\P(w,H)$ where $w$ is a 
periodic word on at most $\ell(m)$ letters and with period $p(m) \leq  {\ell(m)}^m+1$, where $\ell(m)$ is as given in Theorem~\ref{thm:infinitetofinite}. 
\end{theorem}

\begin{proof}
By Theorem~\ref{factors}, $\C$ contains some class $\P(w,H)$ where $w$ is almost periodic and $H$ has at most $\ell(m)$ vertices. 
Take an initial segment $u=u_1\cdots u_{\ell(m)^m+m}$ of the word $w$ of size ${\ell(m)}^m+m$, and consider the $\ell(m)^m+1$ factors $f_1,\dots,f_{\ell(m)^m+1}$ of length $m$ formed by starting at $u_1$, $u_2$, \dots, $u_{\ell(m)^m+1}$. There are $\ell(m)^m$ distinct factors of length $m$ over an alphabet of size $\ell(m)$, and therefore there exist two factors $f_i$ and $f_j$ (with, say, $i<j$) which are identical. We now distinguish two cases: if $j-i \geq m$ (in which case $f_i$ and $f_j$ are disjoint), and if $j-i < m$ (in which case $f_i$ and $f_j$ overlap). 

First suppose $j-i\geq m$, so that the factors $f_i$ and $f_j$ are disjoint.  Let $a=f_i(=f_j)$, and let the portion of the word $u$ between $f_i$ and $f_j$ be $b$. Thus, starting at letter $u_i$, inside $u$ we have a factor of the form $aba$. Since $w$ is almost periodic, it must contain arbitrarily many copies of the factor $aba$, and we claim that $\C$ contains $\P(x,H)$ with $x=(ab)^\infty$. If not, then some forbidden graph $G_i$ of $\C$ is contained in $\P(x,H)$, so there exists a subsequence $i_1<\cdots < i_k$ (where $k=|G|\leq m$) such that $G_{x,H}(i_1,\dots,i_k)\cong G_i$.

Following the proof of~\cite[Theorem 4.6]{atminas:deciding-the-bell:}, we now identify a function $\varphi: \{i_1,\dots,i_k\} \rightarrow \mathbb{N}$ with the following properties: (1) $\varphi(i_j) < \varphi(i_{j'})$ if and only if $i_j < i_{j'}$; (2) $w_{\varphi(i_j)} = x_{i_j}$ for all $j=1,\dots,k$; and (3) $\varphi(i_j) - \varphi(i_{j'}) = 1$ if and only if $i_j - i_{j'}=1$. The existence of such a $\varphi$ will tell us that 
\[G_i \cong G_{x,H}(i_1,\dots,i_k) \cong G_{w,H}(\varphi(i_1),\dots,\varphi(i_k))\in \P(w,H)\subseteq \C,\]
which is a contradiction.

To construct such a $\varphi$, consider a maximal-length sequence $i_j, i_{j+1},\dots,i_{j+p}$ consisting of consecutive integers (i.e.\ $i_{q} +1 = i_{q+1}$ for all $q=j,j+1,\dots,j+p-1$). Since this can have length at most $k\leq m$, we can find it as a factor of $x$ embedded inside some $aba$ factor of $x$. Thus, every maximal block of consecutive integers of $x$ embeds inside $aba$. Now, $aba$ appears infinitely often in $w$, and therefore we can construct $\varphi$ to map the first block of consecutive entries to a factor in the first instance of $aba$ inside $w$, the second block to a factor in the second instance of $aba$, and so on. Thus $\varphi$ is order-preserving (giving condition (1)), preserves letters (condition~(2)) and maps consecutive entries to consecutive entries (which gives (3)).

The second case we consider is where $j-i< m$, in which case the factors $f_i$ and $f_j$ of $u$ overlap. Here, we let $a=u_iu_{i+1}\cdots u_{j-1}$ denote the initial prefix of $f_i$ before $f_j$. After this initial prefix, we repeatedly see copies of $a$ until we find a suffix $a'$ of $f_j$ (which is necessarily a prefix of $a$). This means that the factor $u_iu_{i+1}\cdots u_{j+m-1}$ of $u$ forms a copy of the word $aa\cdots a a'$, where there are at least two $a$s, and where $a'$ is a prefix of $a$. Note that, since $j>i$, this factor $aa\cdots aa'$ is of length at least $m+|a|$, which means that every factor of $a^\infty$ of length $m$ appears in $aa\cdots aa'$.

We now claim that $\C$ contains the class $\P(x,H)$ where $x=a^\infty$. The argument is exactly the same as before upon replacing $aba$ with $aa\cdots aa'$, and we omit the details.

Finally, the maximum length of the period occurs in the case where $f_i$ and $f_j$ do not overlap, in which case the period has length $|ab| \leq \ell(m)^m+1$.
\end{proof}

The above theorem provides a bound on the number of minimal classes of the form $\P(w,H)$ which could be contained in 
$\C=\free(G_1, G_2, \ldots, G_s)$, by bounding the order of the graph $H$ in terms of the sizes of minimal forbidden subgraphs. From this, we can conclude the following.

\begin{corollary}\label{cor:deciding-bell}
There exists a procedure to decide whether $\C=\free(G_1, G_2, \ldots, G_s)$ is above or
 below the Bell numbers, with running time $r(m)$ satisfying
\[ \log\log\log\log \left(r(m)\right) \leq m^{1+o(1)}\] 
where $m=\max\{|G_1|,\dots,|G_s|\}$.
\end{corollary}

\begin{proof}
First, checking whether $\C$ contains one of the minimal classes $\X_1,\dots,\X_{13}$ with infinite distinguishing number can be done in polynomial time, so we assume that $\C$ has $k_\C<\infty$. Now, by Theorem~\ref{thm:infinitetofinite}, if $\C$ lies above the Bell numbers then it contains some class $\P(w,H)$ from a finite list, ranging over all graphs $H$ (with loops allowed) of order at most $\ell_\C\leq\ell(m)$, and all periodic words $w$ whose period is less than $p(m)\leq \ell(m)^m+1$.

For each class $\P(w,H)$ from this list, it suffices to check for each $G_i$ ($i=1,\dots,s$) whether $G_i\in\P(w,H)$. Since $|G_i|\leq m$, if $G_i\in \P(w,H)$ then $G_i$ can be embedded in an initial segment of $w$ of length at most $2mp(m)$.

Thus $\C$ lies below the Bell numbers if and only if every graph $G_{w,H}(1,2,\dots,2mp(m))$ contains some $G_i$ as an induced subgraph, ranging over all graphs $H$ with up to $\ell(m)$ vertices, and all periodic words with period up to length $p(m)$.

To bound the runtime $r = r(m)$, note that there are no more than $p(m)\cdot 2^{\binom{p(m)}{2}}$ possible choices of graph $G_{w,H}(1,2,\dots,2mp(m))$ where $H$ has order at most $\ell(m) \leq p(m)$ and $w$ has period at most $p(m)$. In addition, the total number $s$ of possible forbidden graphs $G_i$ of $\C$ satisfies $s\leq m\cdot 2^{\binom{m}{2}}$. Checking whether $G_i$ is an induced subgraph of $G_{w,H}(1,2,\dots,2mp(m))$ can be completed by considering every subset of size $m$, i.e.~$\binom{2mp(m)}{m}\leq (2mp(m))^m$ instances of the graph isomorphism problem on graphs with at most $m$ vertices, which has runtime at most $2^{O(\sqrt{m\log m})}$ (see Babai, Kantor and Luks~\cite{Babai:computational-complexity:}). 

Thus, we have 
\[r(m) \leq p(m) 2^{\binom{p(m)}{2}} \cdot m 2^{\binom{m}{2}} \cdot (2mp(m))^m \cdot 2^{O(\sqrt{m\log m})}\]
from which we note the dominant term is $2^{\binom{p(m)}{2}}$ and is a quadruple exponential in $m$. The stated bound then follows.
\end{proof}

Together with the results from the previous section we can now conclude with the following decision procedure, which also completes the proof of Theorem~\ref{thm-section-one-statement}.

\begin{corollary}
Let $\C=\free(G_1, G_2, \ldots, G_s)$ be a class with finite distinguishing number and $m=\max\{|G_1|, |G_2|, \ldots, |G_s|\}$.
Then there is an algorithm to decide whether $\C$ is well-quasi-ordered or not, whose running time $r(m)$ satisfies $\log\log\log\log \left(r(m)\right) \leq m^{1+o(1)}$.
\end{corollary}

\begin{proof}
By Theorems~\ref{belowwqo} and~\ref{thm:abovebell}, $\C$ is well-quasi-ordered if and only if it lies below the Bell numbers. However, this procedure has an algorithm with the required running time by Corollary~\ref{cor:deciding-bell}.
\end{proof}


%
%
%
%
%
%
%
%
%
%
\section{Concluding remarks and open problems}\label{sec-conclusion}

\paragraph{Boundary classes} In the previous section, we proved that every class $\P(w,H)$ where $w$ is almost periodic is a boundary class, and consequently defines a minimal antichain in the induced subgraph ordering. This family includes every previously-identified boundary class from~\cite{korpelainen:boundary-properties:}. 

It would be interesting to know what other boundary classes (or, equivalently, minimal antichains) there are. When a class has infinite distinguishing number, it contains at least one of the classes $\X_1,\dots,\X_{13}$. To our minds, the natural first place to start exploring these structures are in classes that contain one of $\X_6$, $\X_7$ or $\X_{13}$, being those that contain vertices that are linearly ordered by their neighbourhood inclusions. 

\paragraph{Periodic minimal antichains} When $\C$ is a finitely defined class with finite distinguishing number, it contains some class $\P(w,H)$ where $w$ is a \emph{periodic} word. Since $\P(w,H)$ is a boundary class, it defines some minimal antichain, being the forbidden graphs of $\P(w,H)$. It seems likely (though we have not proved it here) that the sufficiently large elements of this minimal antichain will have a periodic construction. This would add evidence to support the conjecture that if \emph{any} finitely defined class $\C$ contains some infinite antichain, then it contains a periodic one. 

\paragraph{Labelled well-quasi-ordering}
Lemma~\ref{lemma:lwqo-fd} (Proposition~3 in Daligault, Rao and Thomass\'e~\cite{daligault:well-quasi-orde:}) tells us that every labelled well-quasi-ordered class must be finitely defined. A partial converse to this has been conjectured:
\begin{conjecture}[Korpelainen et al~\cite{korpelainen:boundary-properties:}]
Let $\C$ be a well-quasi-ordered hereditary class. Then $\C$ is finitely defined if and only if it is labelled well-quasi-ordered.
\end{conjecture}
We remark that our work here confirms this conjecture when $\C$ has finite distinguishing number.

\paragraph{Classes below the Bell numbers}
The proof of Theorem~\ref{belowwqo} (either the one given here or the essentially identical argument in~\cite{balogh:the-speed-of-he:,korpelainen:two-forbidden:}) did not use the fact that classes below the Bell numbers are finitely defined to establish the result (instead, this came as a consequence of labelled well-quasi-ordering). Thus, the following question remains open: 

\begin{question}
If $\C=\free(G_1,\dots,G_s)$ is a class below the Bell numbers, can every graph in $\C$ be seen as a $k$-uniform graph where $k$ is bounded by some function of $s$ and $m=\max\{|G_1|, |G_2|, \ldots, |G_s|\}$?
\end{question}

\paragraph{Acknowledgements} We are grateful to Viktor Zamaraev and the anonymous referees for comments on earlier versions of this paper. We must also acknowledge the substantial input of Vadim Lozin to this work, who indicated to the first author that it should be possible to prove that hereditary classes above the Bell numbers are either not well-quasi-ordered, or have infinite distinguishing number.

\bibliographystyle{acm}
\bibliography{refs}

\appendix
\section{Classes below the Bell numbers: a proof of Theorem~\ref{belowwqo}}\label{subsec-belowbell}

Instead of simply making reference to Theorem~30 of~\cite{balogh:the-speed-of-he:}, the starting point we will use is a structural characterisation provided a few pages earlier in the same paper, which builds on constants guaranteed by Lemma~\ref{finitedistinguishingnumber}.

\begin{theorem}[{Balogh et al~\cite[Theorem~28]{balogh:the-speed-of-he:}}] \label{paths}
Let $\C$ be a hereditary class with $k_\C<\infty$, and let $t\geq 5\times 2^{\ell_\C}d_\C$ be fixed.
Then $|\C^n| \geq n^{(1+o(1))n}$ if and only if for every $m$ there
exists a $t$-strong $(\ell_\C, d_\C)$-graph~$G$ in~$\C$ such that its sparsification~$\phi(G)$
has a component of order at least~$m$.
\end{theorem}

Combined with Lemma~\ref{finitedistinguishingnumber}, this says that for any class $\C$ that lies below the Bell numbers, there exist constants $c_\C$, $d_\C$ and $m_C$ such that for any graph $G\in\C$, after removing at most $c_\C$ vertices from $G$ we can form a $t$-strong $(\ell_C,d_C)$-graph. Moreover, for any $(\ell_C,d_\C)$-partition of this graph, the sparsification with respect to this partition has components each of size at most $m_\C$. This is essentially the same statement as Theorem~30 of~\cite{balogh:the-speed-of-he:}.

The other result that is needed is Theorem~2 of \cite{korpelainen:two-forbidden:}, so we now introduce the necessary definitions in order to be able to state it. 

Fix $k\in \mathbb{N}$, let $F_k$ be a graph with $V(F_k) = \{1,\dots,k\}$, and let $M$ be an undirected graph with loops allowed on the same vertex set $\{1,\dots,k\}$.\footnote{In other texts, $M$ is (equivalently) defined to be a $k\times k$ symmetric 0/1 matrix. We have defined it to be another graph here to improve consistency with the construction of $\P(w,H)$ classes.} We construct an infinite graph $H(M,F_k)$ as follows. Let $V(H(M,F_k))=V_1\cup V_2\cup \cdots \cup V_k$ consists of $k$ disjoint copies of $\mathbb N$ (complete with the ordering on $\mathbb N$). Within each set $V_i$, we form a clique if there is a loop on vertex $i$ in the graph $M$, and an independent set if there is no loop. Between each pair of sets $V_i$ and $V_j$ with $i\neq j$, the graphs $M$ and $F_k$ give four possibilities:
\begin{itemize}
\item $ij\in E(F_k)$, $ij\not\in E(M)$: matching between $V_i$ and $V_j$ (connecting the $\ell$th vertex of $V_i$ to the $\ell$th vertex of $V_j$);
\item $ij \in E(F_k)$, $ij\in E(M)$: co-matching between $V_i$ and $V_j$ (the $\ell$th vertex of $V_i$ and the $\ell$th vertex of $V_j$ are co-connected);
\item $ij \not\in E(F_k)$, $ij\not\in E(M)$: $V_i$ and $V_j$ are co-joined;
\item $ij \not\in E(F_k)$, $ij\in E(M)$: $V_i$ and $V_j$ are joined.
\end{itemize}
Informally, the graph $M$ determines whether there is a high or low density of edges between $V_i$ and $V_j$, and the graph $F_k$ then specifies whether this relationship is a matching/co-matching or joined/co-joined. 

We say that a (finite) graph $G$ is \emph{$k$-uniform} if it is an induced subgraph of $H(M,F_k)$ for some pair of graphs $M$ and $F_k$. We will require the following two facts about $k$-uniform graphs:

\begin{lemma}[{Korpelainen and Lozin~\cite[Lemma 2]{korpelainen:two-forbidden:}}]\label{kuniformadd}
If a graph $G$ has a subset $W$ of at most $c$ vertices such that $G-W$ is k-uniform, then $G$ is $(2^c(k+1)-1)$-uniform.  
\end{lemma}

\begin{theorem}[{Korpelainen and Lozin~\cite[Theorem 2]{korpelainen:two-forbidden:}}]\label{kuniformwqo}
For any fixed $k$, the set of $k$-uniform graphs is labelled-well-quasi-ordered.
\end{theorem}

The following technical lemma provides the connection to $(\ell,d)$-graphs that we require.

\begin{lemma} \label{kuniformld}
Let $t\geq 5\times 2^{\ell}d$, and let $G$ be a $t$-strong $(\ell,d)$-graph. If each component in the sparsification $\phi(G)$ of $G$ contains at most $m$ vertices, then $G$ is a $k$-uniform graph for $k \leq m \ell^m 2^{\binom{m}{2}+1}$.
\end{lemma}
\begin{proof}
Since $G$ is a $t$-strong $(\ell,d)$ graph for $t\geq 5\times 2^\ell d$, the sparsification $\phi(G)$ is unique irrespective of the chosen $(\ell,d)$ partition. Thus, fix any $(\ell,d)$ partition $\pi$ of $G$, and consider the connected components of $\phi(G)$, all of which contain at most $m$ vertices. 

We define the following equivalence relation: two components $C$ and $C'$ of $\phi(G)$ are \emph{$\pi$-equivalent} if there is a graph isomorphism $f: C \rightarrow C'$ such that for every vertex $v \in C$ both $v$ and $f(v)$ belong to the same bag $V_i$ of the partition $\pi$. 

Set $V'$ to be a subset of vertices of $V(G)$ which in $\phi(G)$ induces exactly one component from each $\pi$-equivalence class. There are fewer than $2^{\binom{m}{2}+1}$ graphs on at most $m$ vertices, and each such graph can be embedded into bags of the partition $\pi$ in at most $\ell^m$ different (non-$\pi$-equivalent) ways, therefore (since each graph has at most $m$ vertices)  $|V'| \leq  m \ell^m 2^{\binom{m}{2}+1}$.

Now set $k=|V'|$, and let $F_k$ be the subgraph of $\phi(G)$ induced on the set $V'$. For two vertices $u, v \in V'$, $u \in V_i$, $v \in V_j$ set $uv\not\in E(M)$ if $(V_i,V_j)$ is $d$-sparse, and $uv\in E(M)$ if $(V_i, V_j)$ is $d$-dense. Now, by the definition of $k$-uniformity, $G$ is an induced subgraph of $H(M,F_k)$, as required.
\end{proof}

The proof of Theorem~\ref{belowwqo} will be completed by the next theorem, when combined with Theorem~\ref{kuniformwqo}.

\begin{theorem}\label{thm:below-bell-k-uniform}Let $\C$ be a hereditary class below the Bell numbers. Then there exists a constant $k$ such that every graph in $\C$ is a $k$-uniform graph.\end{theorem}

\begin{proof}
Let $\C$ be a hereditary class below the Bell numbers, necessarily therefore of finite distinguishing number $k_{\C}$. Given any $G \in {\C}$, by Lemma~\ref{finitedistinguishingnumber} we know that by removing at most $c_\C$ vertices from $G$ we can find a $t$-strong $(\ell_C, d_C)$ induced subgraph $G'$, where $t=5\times 2^{\ell_\C} d_\C$. Fix any $(\ell_C,d_\C)$ partition $\pi$ of $G'$.

By Theorem~\ref{paths} there is an absolute constant $m_{\C}$ depending only on $\C$ such that the sparsification of $G'$ has every component of size at most $m_{\C}$. 
From Lemma~\ref{kuniformld} it follows that $G'$ is a $k'$-uniform graph with $k' = m_{\C} \ell_{\C}^{m_{\C}} 2^{\binom{m_{\C}}{2}+1}$. Since $G$ contains at most $c_\C$ more vertices than $G'$, Lemma~\ref{kuniformadd} tells us that $G$ is $k$-uniform, where $k=2^{c_\C}(k'+1)-1$. 
\end{proof}

\end{document}